\documentclass[12pt]{article}

\usepackage{amsmath,amsthm,amssymb,graphicx,tikz,mathtools,enumitem}
\usepackage{tikz-cd}
\usepackage[margin=1in]{geometry}
\usepackage{mathrsfs}
\usepackage[all]{xy}
\usepackage{amsfonts}
\usepackage{latexsym}
\usepackage{subfigure}
\usepackage{fullpage}
\usepackage{float}
\usepackage{xypic}
\usepackage{bbm}
\usepackage{xcolor}

\usepackage[percent]{overpic}

 \DeclareMathOperator{\rk}{rk}

\DeclareMathOperator{\Kh}{Kh}
\DeclareMathOperator{\Lee}{Lee}

\DeclareMathOperator{\CKh}{CKh}

\DeclareMathOperator{\HFK}{HFK}
\DeclareMathOperator{\HFL}{HFL}
\DeclareMathOperator{\SFH}{SFH}
\DeclareMathOperator{\BS}{BS}

\DeclareMathOperator{\Spin}{Spin}

\newcommand{\Z}{\mathbb{Z}}

\newcommand{\Q}{\mathbb{Q}}

\newtheorem{thm}{Theorem}
\newtheorem{prop}[thm]{Proposition}

\newtheorem{cor}[thm]{Corollary}

\theoremstyle{definition}
\newtheorem{rem}[thm]{Remark}

\usepackage{hyperref}
\hypersetup{pdftex,colorlinks=true,allcolors=blue}

\begin{document}
\title{Khovanov homology detects T(2,6)}
\author{Gage Martin}
\date{}
\maketitle

\begin{abstract}
We show if $L$ is any link in $S^3$ whose Khovanov homology is isomorphic to the Khovanov homology of $T(2,6)$ then $L$ is isotopic to $T(2,6)$. We show this for unreduced Khovanov homology with $\mathbb{Z}$ coefficients.
\end{abstract}
\section{Introduction}

Khovanov homology is a combinatorially defined, bi-graded $R$ module $\Kh^{i,j}(L, R)$ which is associated to an oriented link $L \subseteq S^3$~\cite{khovanov_categorification_2000}. The graded Euler characteristic of Khovanov homology is the Jones polynomial. Many of the topological applications of Khovanov homology come from algebraic relationships to Floer homologies either implicitly through an analogy (e.g. the definition of the $s$-invariant) or explicitly through a spectral sequence.

In the spirit of finding connections between topological information and $\Kh^{i,j}(L, R)$ is the question of detection. Specifically, Khovanov homology is said to detect a link $L_0$ if given any link $L$ then $L$ is isotopic to $L_0$ if and only if $\Kh^{i,j}(L, R)\cong  \Kh^{i,j}(L_0, R)$. Kronheimer and Mrowka showed that Khovanov homology detects the unknot~\cite{kronheimer_khovanov_2011}. Khovanov homology is also known to detect the unlink~\cite{hedden_khovanov_2013}~\cite{batson_link-splitting_2015}, the Hopf link~\cite{baldwin_khovanov_2019}, the trefoil~\cite{baldwin_khovanov_2018}, the connected sum of two Hopf links~\cite{xie_classification_2019}, the torus link $T(2,4)$~\cite{xie_classification_2019}, and split links~\cite{lipshitz_khovanov_2019}.

In this paper we prove an additional detection result for Khovanov homology

\newtheorem*{detect}{Theorem~\ref{thm:DetectsT(2,6)}}
\begin{detect}
Let $L$ be a link with $\Kh(L,\Z) \cong \Kh(T(2,6),\Z)$. Then $L$ is isotopic to $T(2,6)$.
\end{detect}

The proof of Theorem~\ref{thm:DetectsT(2,6)} is similar in spirt to the proof in~\cite{xie_classification_2019} that Khovanov homology detects $T(2,4)$ but uses Dowlin's spectral sequence to knot Floer homology, rather than the Kronheimer-Mrowka spectral sequence to singular instanton Floer homology.

\subsection*{Acknowledgements}
The author would like to thank John Baldwin, Eli Grigsby, Siddhi Krishna, and Ian Montague for helpful conversations.

\section{Background}

\subsection{Khovanov homology}

Khovanov homology is a combinatorially defined invariant that assigns to an oriented link $L \subseteq S^3$ a bi-graded $R$-module $\Kh^{i,j}(L, R)$ which is the homology of a chain complex $\CKh(D)$ associated to a diagram $D$ for $L$~\cite{khovanov_categorification_2000}. The $i$ grading is the homological grading and the $j$ grading is the quantum grading.

A choice of a basepoint $p \in L$ defines an action of $R[X]/X^2 = 0$ on $\CKh(D)$. Quotienting by the image of this action and then taking homology gives rise to reduced Khovanov homology $\widetilde{\Kh}^{i,j}(L, R)$. The rank of $\widetilde{\Kh}^{i,j}(L, \mathbb{F}_2)$ is exactly half the rank of $\Kh^{i,j}(L, \mathbb{F}_2)$~\cite[Corollaries 3.2.B-C]{shumakovitch_torsion_2014}.

Multiple spectral sequences starting at Khovanov homology and converging to other homology theories have been constructed. We briefly recall the spectral sequences that will be needed in the proof of Theorem~\ref{thm:DetectsT(2,6)}.

Using a similar construction to Khovanov homology, Lee defined an invariant of an oriented link $L \subseteq S^3$ called Lee homology $\Lee^{i}(L, \Q)$, and from the construction there is a spectral sequence from $\Kh^{i,j}(L, \Q) $ to $\Lee^{i}(L, \Q)$. Lee showed that the total rank of $\Lee^{i}(L, \Q)$ for an $n$-component link $L$ is $2^n$ and showed an explicit bijection between generators of $\Lee^{i}(L, \Q)$ and choices of orientations of $L$. The homological gradings in which $\Lee^{i}(L, \Q)$ is non-zero are determined by the pairwise linking numbers of the different components of $L$~\cite{lee_endomorphism_2005}.

Batson-Seed constructed a link splitting spectral sequence from Khovanov homology. To simplify the exposition, we will restrict to the case that $L = K_1 \cup K_2$ is a 2-component link which is what is relevant to the proof of Theorem~\ref{thm:DetectsT(2,6)}. If $L$ has two components, then the Batson-Seed construction gives a spectral sequence from $\Kh^{i-j}(L, \mathbb{F}_2)$ to a homology theory $\BS^{i-j}(L,\mathbb{F}_2)$ where $\BS^{i-j+A}(L,\mathbb{F}_2) \cong \Kh^{i-j}(K_1 \sqcup K_2) \cong \Kh^{i-j}(K_1) \otimes \Kh^{i-j}( K_2)$ where $A$ is some overall shift determined by $L$~\cite{batson_link-splitting_2015}.   

Pointed Khovanov homology is a generalization of reduced Khovanov homology to a link $L$ with a set of base points $p_1 ,\ldots , p_m \in L$ and a corresponding action of $R[X_1, \ldots , X_m] / X_1^2 = \cdots = X_m^2 = 0$ on $\CKh(L)$ for each base point~\cite{baldwin_khovanov_2017}. Dowlin constructed a spectral sequence from relatively $\delta' = j - 2i$ graded pointed Khovanov homology to relatively $\delta' = 2M - 2A$ graded knot Floer homology~\cite{dowlin_spectral_2018}.

\subsection{Knot Floer homology, link Floer homology, and sutured Floer homology}

Knot Floer homology is an invariant that assigns to an oriented link $L \subseteq S^3$ a bi-graded $R$-module $\widehat{\HFK}(L)$. The two gradings are the Maslov grading $M$ and the Alexander grading~$A$~\cite{ozsvath_holomorphic_2004}.

Link Floer homology $\widehat{\HFL}(L)$ is a generalization of knot Floer homology which is graded by an Alexander grading $a_i$ for each component of $L$ in addition to the Maslov grading~$M$~\cite{ozsvath_holomorphic_2008}.

To recover knot Floer homology from link Floer homology for an $n$-component link $L$, start with $\widehat{\HFL}(L)$ and define a single Alexander grading as the sum over the Alexander gradings of all components, $A = \sum a_i$. Then $\widehat{\HFL}(L)$ graded by $A$ and $M + \frac{n - 1}{2}$    is isomorphic to $\widehat{\HFK}(L)$~\cite{ozsvath_holomorphic_2008}.

Sutured Floer homology $\SFH(M)$ is a version of Heegaard Floer homology defined for a balanced sutured manifold $M$. The homology $\SFH(M)$ splits over relative $\Spin^C$ structures on $M$~\cite{juhasz_holomorphic_2006}. 

 The complement of a link $S^3 \setminus L$ is naturally a balanced sutured manifold. The sutures are pairs of two oppositely oriented meridional sutures on each component of the boundary. The sutured Floer homology $\SFH(S^3 \setminus L)$ is isomorphic to $\widehat{\HFL}(L)$ with the relative $\Spin^C$ structures corresponding to the multi-Alexander gradings.

Given a properly embedded oriented surface with boundary $S$ in a sutured manifold $M$, which satisfies some technical conditions about how $\partial S$ intersects the the sutures of $M$, $S$ defines a sutured manifold decomposition from $M$ to $M' = M \setminus  \textrm{Int(N}(S)$. Juh\'{a}sz showed that $\SFH(M')$ is isomorphic to the direct summands of $\SFH(M)$ corresponding to ``outer" $\Spin^C$ structures~\cite{juhasz_floer_2008}.

\subsection{Link Floer homology detects braids} 

An argument similar in spirit to arguments in~\cite[Theorem~1.5]{juhasz_floer_2008} and~\cite{grigsby_sutured_2014} shows that link Floer homology detects braids in the complement of a fibered component. This braid detection result is known to some experts but the author is unaware of a proof in the literature so one is produced here. We provide a proof a more general statement is needed in the proof of Theorem~\ref{thm:DetectsT(2,6)}. A version of the following argument was communicated to the author by John Baldwin~\cite{BaldwinPersonalCommunication}. For a definition of a braid in the complement of a fibered knot refer to~\cite[Definition~1.2]{kawamuro_self-linking_2011}.

\begin{prop}
Suppose $L\subseteq M$ is a link with $l$ components with a fibered component $K$ and $M \setminus L$ is irreducible. Then $L \setminus K$ is a braid in the complement of $K$ if and only if $\widehat{\HFL(L)}$ has rank $2^{l-1}$ in the highest (and lowest) non-zero Alexander grading associated to $K$.
\end{prop}

\begin{proof}

Consider a fiber surface $S$ bounded by $K$ which intersects $L \setminus K$ minimally. Cutting open the sutured manifold $M\setminus L $ along $S \setminus L$ gives a new sutured manifold $N$. The sutured Floer homology of $N$ is isomorphic to the link Floer homology of $L$ supported in a constant $a_K$ grading of $\frac{1}{2} c(S,t) = \chi(S) + I(S) - r(S,t)$, where $a_K$ is the Alexander grading associated to $K$~\cite[Theorem~3.11]{juhasz_floer_2008}. For a definition of $\frac{1}{2} c(S,t)$, see~\cite[Definition~3.8]{juhasz_floer_2008}.

The sutured manifold $N$ contains $l-1$ pairs of parallel sutures corresponding to the base points on the components of $L\setminus K$. Removing these superfluous pairs of sutures gives a new sutured manifold $N'$ and $\rk(\SFH (N)) = 2^{l-1} \rk(\SFH(N'))$. Finally, $\rk(\SFH(N')) = 1$ if and only if $N'$ is a product sutured manifold~\cite[Prop~9.4]{juhasz_holomorphic_2006}~\cite[Theorem~1.4]{juhasz_floer_2008}. The manifold $N'$ is a product sutured manifold exactly when $L \setminus K$ is a braid in the complement of $K$.

To see that $c(S,t)$ is the lowest non-zero $a_K$ grading, consider increasing the genus of the Seifert surface for $K$ by adding $h$ handles to the genus $g$ surface $S$ in the complement of $L$ to obtain a new surface $S'$ of genus $g + h$. Then the sutured manifold obtained by cutting open along $S'$ is not taut if $h \geq 1$ so the link Floer homology in $a_U$ grading $\frac{1}{2}c(S',t)$ is zero~\cite[Prop~9.8]{juhasz_holomorphic_2006} and one can compute that $c(S',t) = c(S,t) - 2h$.

The rank in the lowest non-zero Alexander grading associated to $K$ is the same as the rank in the highest non-zero Alexander grading associated to $K$ because of the symmetry of Link Floer homology.
\end{proof}

Taking $L \subseteq S^3$ and the fibered knot to be the unknot gives the following corollary.

\begin{cor}\label{LFL:braids}
Suppose $L\subseteq S^3$ is a link with $l$ components with an unknotted component $U$ and each component of $L \setminus U$ has non-zero geometric linking with $U$. Then $L \setminus U$ is a braid in the complement of $U$ if and only if $\widehat{\HFL(L)}$ has rank $2^{l-1}$ in the highest (and lowest) non-zero Alexander grading associated to $U$.\end{cor}

\begin{rem}
For the case with the unknot, if $D$ intersects $L$ in $n$ points then a simple computation of $c(D,t)$ shows that the highest non-zero Alexander grading will be $n/2$.
\end{rem}

\section{Khovanov Homology detects T(2,6)}

\begin{table}[]
\centering
\begin{tabular}{|c|c|c|c|c|c|c|c|}
\hline
         & $i = 0$      & $i = 1$ & $i = 2$      & $i = 3$                    & $i = 4$      & $i = 5$                    & $i = 6$      \\ \hline
$j = 18$ &              &         &              &                            &              &                            & $\mathbb{Z}$ \\ \hline
$j = 16$ &              &         &              &                            &              & $\mathbb{Z}$               & $\mathbb{Z}$ \\ \hline
$j = 14$ &              &         &              &                            &              & $\mathbb{Z} / 2\mathbb{Z}$ &              \\ \hline
$j = 12$ &              &         &              & $\mathbb{Z}$               & $\mathbb{Z}$ &                            &              \\ \hline
$j = 10$ &              &         &              & $\mathbb{Z} / 2\mathbb{Z}$ &              &                            &              \\ \hline
$j = 8$  &              &         & $\mathbb{Z}$ &                            &              &                            &              \\ \hline
$j = 6$  & $\mathbb{Z}$ &         &              &                            &              &                            &              \\ \hline
$j = 4$  & $\mathbb{Z}$ &         &              &                            &              &                            &              \\ \hline
\end{tabular}
\caption{The Khovanov homology of the torus link T(2,6) computed using SageMath~\cite{sagemath}}\label{KhT(2,6)}
\end{table}

In this section, we show that Khovanov homology detects the torus link T(2,6). For reference, the Khovanov homology is shown in Table~\ref{KhT(2,6)}.

\begin{thm}\label{thm:DetectsT(2,6)}
Let $L$ be a link with $\Kh(L,\Z) \cong \Kh(T(2,6),\Z)$, then $L$ is isotopic to $T(2,6)$.
\end{thm}

Theorem~\ref{thm:DetectsT(2,6)} follows from the two propositions below.

\begin{prop}\label{FirstProperties} 
If $\Kh(L,\Z) \cong \Kh(T(2,6),\Z)$, then $L$ is a 2-component link with linking number $3$ and each of the components is an unknot.
\end{prop}

\begin{prop}\label{Braided}
If $\Kh(L,\Z) \cong \Kh(T(2,6),\Z)$, then one component of $L$ is a braid in the complement of the other component.
\end{prop}

\begin{proof}[Proof of Theorem~\ref{thm:DetectsT(2,6)} from Propositions~\ref{FirstProperties} and~\ref{Braided}]
From Propositions~\ref{FirstProperties} and~\ref{Braided}, $L$ must be $\widehat{\beta} \cup U$ where $\beta$ is a 3-braid whose closure is an unknot and $U$ is the braid axis.

Up to isotopy in the complement of the braid axis, there are only three possible 3-braids whose closures are the unknot, $\sigma_1 \sigma_2$, $\sigma_1^{-1} \sigma_2^{-1}$ and $\sigma_1\sigma_2^{-1}$ so $L$ must be one of these braids together with its braid axis. The first two possibilities both represent $T(2,6)$. The final possibility using the braid $\sigma_1\sigma_2^{-1}$ gives the link L6a2 and $\Kh(\textrm{L6a2},\Z) \not \cong \Kh(T(2,6),\Z)$ because they have different ranks~\cite{livingston_linkinfo_nodate}.
\end{proof}

\begin{proof}[Proof of Proposition~\ref{FirstProperties}]
The fact that $\Kh(L,\Z) \cong \Kh(T(2,6),\Z)$ means that $\Kh(L,\Z)$ is supported in even quantum gradings and so $L$ has an even number of components because the non-zero quantum gradings of $\Kh(L,R)$ agrees mod 2 with the number of components of $L$.

The Lee homology of $L$ has even rank in each homological grading and has total rank $2^n$ where $n$  is the number of components of $L$. So then rank inequalities from the spectral sequence between Khovanov homology and Lee homology show that $L$ has exactly two components because there are only two homological gradings where the rank of Khovanov homology is more than 1 and in each of these gradings the rank is exactly 2. Furthermore, these homological gradings are $i = 0$ and $i = 6$ so the linking number of the two components is 6/2 = 3~\cite[Proposition~4.3]{lee_endomorphism_2005}.

Considering the Batson-Seed spectral sequence over $\mathbb{F}_2$ from $\Kh(L)$ to $\Kh(L')$ where $L'$ is the split link comprised of the two components of $L$~\cite[Theorem~1.1]{batson_link-splitting_2015}. The total rank of $\Kh(L,\mathbb{F}_2)$ is 12 and the total rank of $\Kh(L')$ is the product of the ranks of the Khovanov homology of the two components. Additionally, over $\mathbb{F}_2$, the rank of Khovanov homology of a knot over $\mathbb{F}_2$ must be twice an odd number because it is twice the rank of reduced Khovanov homology over $\mathbb{F}_2$ which always has odd rank for a knot. 

Then the only possible ranks for the Khovanov homologies of the components of $L$ are 2 and 6. Then the only possibilities for the components of $L$ are either two unknots or an unknot and a trefoil because the unknot is the only knot whose Khovanov homology has rank two over $\mathbb{F}_2$~\cite[Theorem 1.1]{kronheimer_khovanov_2011} and the trefoil is the only knot whose Khovanov homology has rank 6 over $\mathbb{F}_2$~\cite[Theorem~1.4]{baldwin_khovanov_2018}. Examining the rank of Khovanov homology of $L$ over $\Q$ in each $i - j$ grading, which is preserved by the Batson-Seed spectral sequence up to an overall shift, rules out the possibility that one of the components of $L$ is a trefoil because there is no overall shift possible to make the ranks agree with the ranks in $i-j$ gradings of the tensor product $\Kh(U) \otimes \Kh(T)$ where $U$ and $T$ are the unknot and trefoil respectively. \end{proof}

\begin{proof}[Proof of Proposition~\ref{Braided}] 
To show that one component of $L$ is braided with respect to the other, we will use the spectral sequence from the pointed Khovanov homology of $L$ to a singly graded version of knot Floer homology constructed by Dowlin~\cite{dowlin_spectral_2018}. In this proof, we will use $\delta'$ to refer to both grading Khovanov homology by $j - 2i$ and grading knot Floer homology by $2M - 2A$. We will use $\delta$ to refer to grading knot Floer homology by $A - M$.

From knowing $\Kh(L,\Z)$ we can see that the reduced Khovanov homology of $L$ over $\mathbb{F}_2$ is rank 6. So then the reduced Khovanov homology of $L$ over $\Q$ has rank no greater than 6. Because $L$ is a 2 component link, the pointed Khovanov homology $\Kh(L , \textbf{p} )$ over $\Q$ has rank no greater than 12 and the fact that $L$ is Khovanov thin implies that $ \Kh(L , \textbf{p} )$ is supported in a single $\delta' = j - 2i$ grading~\cite[Lemma~2.11]{baldwin_khovanov_2017}.

The Dowlin spectral sequence preserves the relative $\delta'$ grading so $\widehat{\HFK}(L)$ is supported in a single $\delta = - 1/2 \delta ' $ grading.

Now we consider $\widehat{\HFL}(L)$ in order to show that one component is a braid in the complement of the other component. By Corollary~\ref{LFL:braids}, we want to show that in the top non-zero grading of either $a_1$ or $a_2$ the rank of $\widehat{\HFL}(L)$ is exactly two.

The fact that each of the two components of $L$ is an unknot and the existence of the spectral sequence from $\widehat{\HFL}(L = K_1 \cup K_2)$ to $\widehat{\HFL}(K_1) \otimes V$ implies that $\widehat{\HFL}(L)$ has rank at least 1 in the gradings $M = 0 , a_1 = 3/2$ and $M = - 1 , a_1 = 3/2$ where each generator sits in some unknown $a_2$ grading~\cite[Lemma~2.4]{baldwin_categorified_2015}. Similarly $\widehat{\HFL}(L)$ has rank at least 1 in the gradings $M = 0 , a_2 = 3/2$ and $M = - 1 , a_2 = 3/2$ where each generator sits in some unknown $a_1$ grading. The fact that $\widehat{\HFL}(L)$ is supported in a single $\delta= a_1 + a_2 - M - 1/2$ grading allows us to write the unknown gradings in terms of a variable $x$. There is at least one generator in $(M , a_1 , a_2)$ gradings $(0, 3/2 , x)$ and $(-1 , 3/2 , x-1)$, where these generators survive in the spectral sequence induced by $a_2$. Also there is at least one generator in  $(0, x, 3/2 )$ and $(-1  , x-1, 3/2 )$, where these generators survive in the spectral sequence induced by $a_1$.

The symmetry of $\widehat{\HFL}(L)$~\cite[Proposition~8.2]{ozsvath_holomorphic_2008} then tells us that $\widehat{\HFL}(L)$ also has rank at least 1 in the following four gradings $(-3 - 2x , -3/2 , -x )$, $(-3 - 2x+1 , -3/2 , -x+1  )$, $(-3 - 2x, -x , -3/2  )$, $(-3 - 2x+1 , -x+1 , -3/2  )$.

From here the proof breaks into seven cases depending on the value of $x$. There is a case where $x > 5/2$, a case where $x < -3/2$ and a case for each of the following values $5/2 , 3/2 , 1/2, -1/2 , -3/ 2$. For each case we deduce that one component of $L$ is a braid in the complement of the other component.

We first address the case $x = 3/2$ which is the case that occurs if $L = T(2,6)$. If the Dowlin spectral sequence was known to preserve the absolute $\delta'$ grading then this would be the only case that needed to be considered.

\emph{The case $x = 3/2$} 

Setting $x = 3/2$ we have that $\widehat{\HFL}(L)$ has rank at least 1 in the tri-gradings $( 0 ,3/2,  3/2)$, $(- 1 ,  1/2, 3/2)$ and $(- 1 ,  3/2, 1/2)$ and these generators survive in one or both of the spectral sequences induced by $a_i$. Additionally, $\widehat{\HFL}(L)$ also has rank at least 1 in gradings $(-6 ,  -3/2, -3/2)$, $(-5,  -3/2, -1/2)$,  and $(-5, -1/2,  -3/2)$ and these generators do not survive in either spectral sequence. The partial complex with these six generators is shown in Figure~\ref{6gens}.

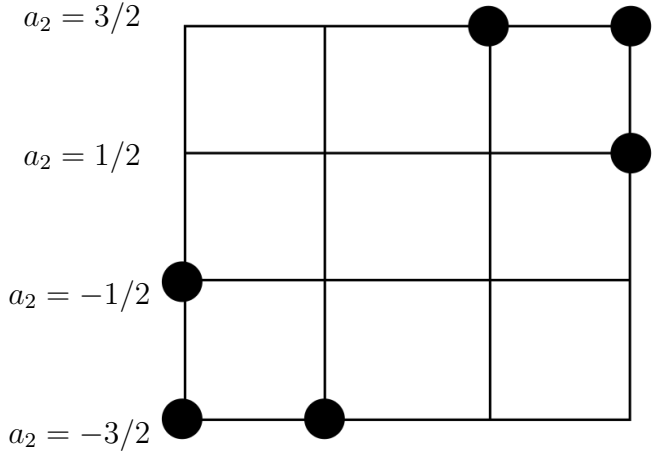
\begin{figure}
\centering{
\begin{overpic}[width= 0.4\textwidth]{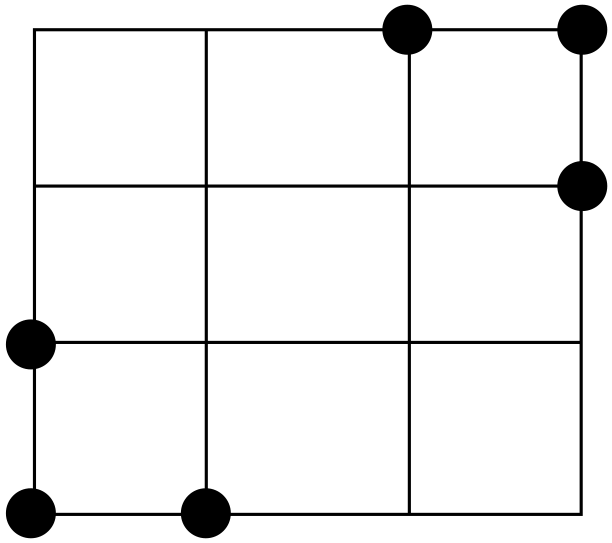}
\put(85,100){$a_1 = 3/2$}
\put(55,100){$a_1 = 1/2$}
\put(20,100){$a_1 = -1/2$}
\put(-15, 100){$a_1 = -3/2$}
\put(-27, 84){$a_2 = 3/2$}
\put(-27, 56){$a_2 = 1/2$}
\put(-30, 28){$a_2 = -1/2$}
\put(-30, 0){$a_2 = -3/2$}
\end{overpic}
\caption{A partial $\widehat{\HFL}(L)$ complex when $x = 3/2$ with 6 generators. The dots represent bi-degrees where the partial complex has rank one.}\label{6gens}
}
\end{figure}

At this point, there are at most 6 more generators we can add to construct a possible $ \widehat{\HFL}(L)$. When we have finished adding generators to a possible $ \widehat{\HFL}(L)$, the end result must have even rank in every $a_1$ grading and every $a_2$ grading, otherwise it is impossible to have spectral sequences to $\widehat{\HFL}(K_1) \otimes V$ and $\widehat{\HFL}(K_2) \otimes V$.  Additionally, it must have the symmetry that $\widehat{\HFL}_{M}(L, a_1 , a_2 ) \cong \widehat{\HFL}_{M - a_1 - a_2 }(L, -a_1 , -a_2 ) $~\cite[Proposition~8.2]{ozsvath_holomorphic_2008}, and it must be possible to add differentials changing the Maslov index by 1 so that there are spectral sequences  induced by each $a_i$ grading with the $E_\infty$ pages mentioned in the previous paragraph.

If we add generators with $a_i$ grading larger than $3/2$ then the above requirements about even parity in every grading and symmetry mean that we must add exactly two generators in that grading (which is now the top grading) and so one of the components is a braid in the complement of the other. So now we can assume that we only add generators whose $a_i$ gradings are less than or equal to $3/2$ in absolute value. 

Having the appropriate $E_\infty$ pages of the two spectral sequences now means that there must be a generator in grading $M = -4 , a_1 = -1/2, a_2 = -1/2$ and one in grading $M = -2, a_1 = 1/2, a_2 = 1/2$. The partial complex with these generators is shown in Figure~\ref{8gens}.

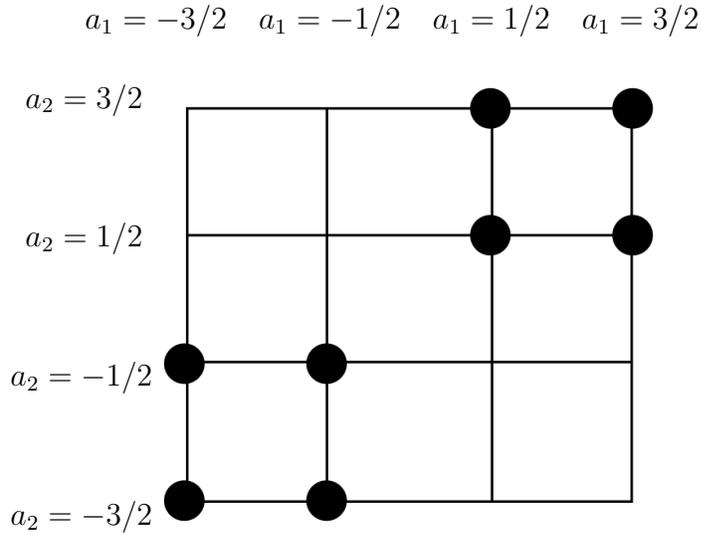
\begin{figure}
\centering{
\begin{overpic}[width= 0.4\textwidth]{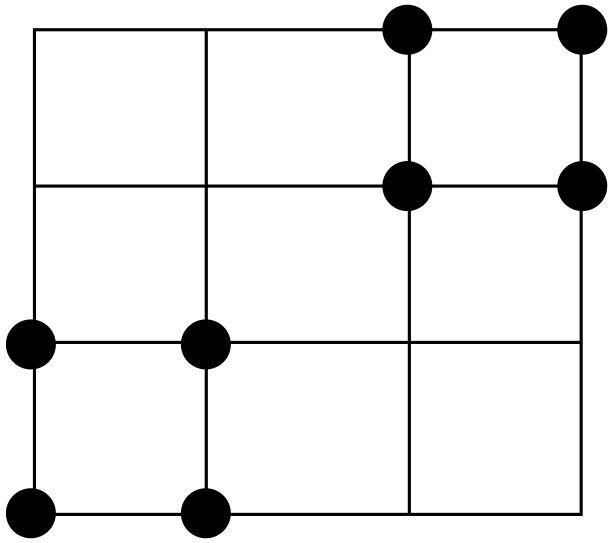}
\put(85,100){$a_1 = 3/2$}
\put(55,100){$a_1 = 1/2$}
\put(20,100){$a_1 = -1/2$}
\put(-15, 100){$a_1 = -3/2$}
\put(-27, 84){$a_2 = 3/2$}
\put(-27, 56){$a_2 = 1/2$}
\put(-30, 28){$a_2 = -1/2$}
\put(-30, 0){$a_2 = -3/2$}
\end{overpic}
\caption{A partial $\widehat{\HFL}(L)$ complex when $x = 3/2$ with 8 generators if every generator has $a_i$ grading no greater than $  3/2$ in absolute value. The dots represent bi-degrees where the partial complex has rank one.}\label{8gens}
}
\end{figure}

Now there are at most 4 more generators to add. If the end result will have that neither component is a braid in the complement of the other then two those generators must be added at either $a_1 = 3/2, a_2 = -3/2$ or  $a_1 = 3/2, a_2 = 3/2$ and the other two are then added in the appropriate place for symmetry. There are three possible ways to add the two generators, either both in  $a_1 = 3/2, a_2 = -3/2$, both in $a_1 = 3/2, a_2 = 3/2$, or one in $a_1 = 3/2, a_2 = -3/2$ and the other in $a_1 = 3/2, a_2 = 3/2$. For each of these ways of adding generators, it is impossible to add differentials that give the desired $E_\infty$ pages of both spectral sequences. So then the link $L$ must have that one of its components is a braid in the complement of the other.

\emph{The case $x = 5/2$ }

Setting $x = 5/2$ we have that $\widehat{\HFL}(L)$ has rank at least 1 in the tri-gradings $( 0, 5/2,  3/2)$, $(0 ,  3/2, 5/2)$ and $(- 1 ,  3/2, 3/2)$ and these generators survive in one or both of the spectral sequences induced by $a_i$. Additionally, $\widehat{\HFL}(L)$ also has rank at least 1 in gradings $(-6 ,  -5/2, -3/2)$, $(-6,  -3/2, -5/2)$,  and $(-5, -3/2,  -3/2)$ and these generators do not survive in either spectral sequence. The partial complex with these six generators is shown in Figure~\ref{xfivehalf6gens}.

\begin{figure}
\centering{
\begin{overpic}[width= 0.4\textwidth]{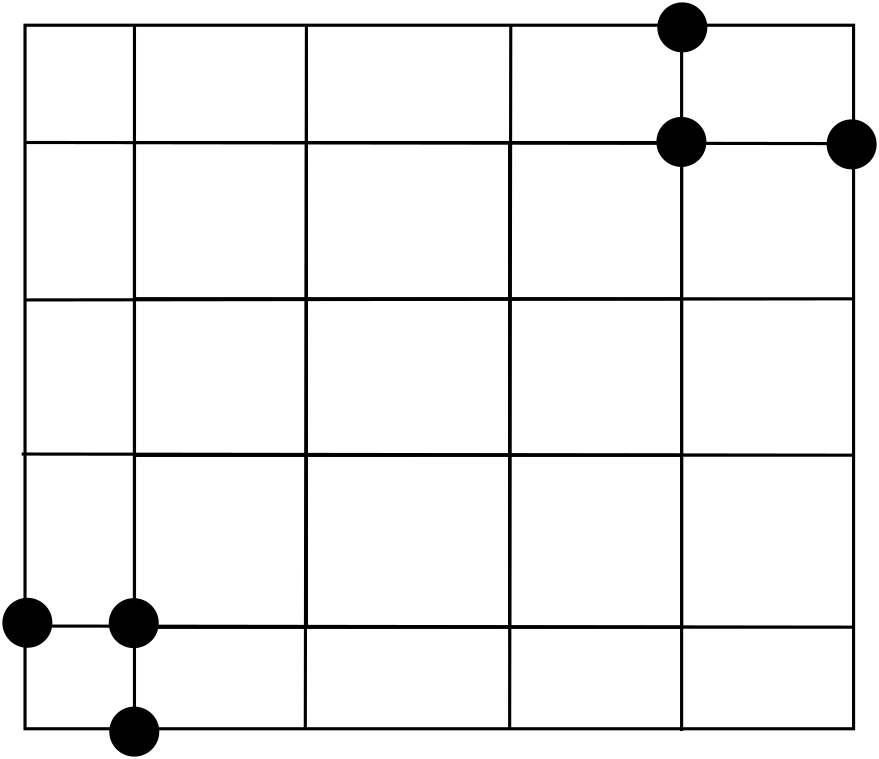}
\put(85,90){$a_1 = 5/2$}
\put(-15, 90){$a_1 = -5/2$}
\put(-27, 80){$a_2 = 5/2$}
\put(-27, 68){$a_2 = 3/2$}
\put(-27, 50){$a_2 = 1/2$}
\put(-30, 32){$a_2 = -1/2$}
\put(-30, 13){$a_2 = -3/2$}
\put(-30, 0){$a_2 = -5/2$}
\end{overpic}
\caption{A partial $\widehat{\HFL}(L)$ complex when $x = 5/2$ with 6 generators. The dots represent bi-degrees where the partial complex has rank one. The vertical columns represent $a_1$ gradings with grading increasing by 1 from $-5/2$ to $5/2$.}\label{xfivehalf6gens}
}
\end{figure}
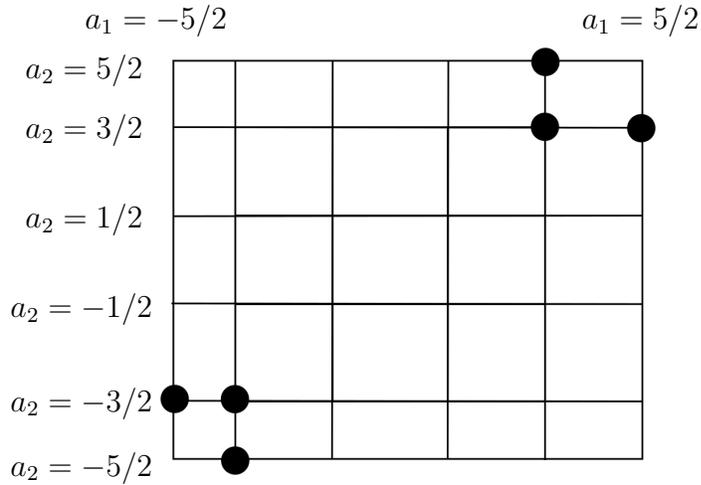

Generators must be added to the gradings $a_i = \pm 5/2$ to allow for the desired spectral sequences to exist. There are two possible ways to do this while preserving the symmetry, the first is adding generators at the $(a_1 , a_2)$ gradings $(5/2, 1/2)$ and $(1/2 , 5/2)$. Then two more generators must be added to preserve symmetry leaving a complex with 10 generators and only being able to add up to two additional generators. The partial complex has exactly two generators in the maximal $a_i$ grading for $i = 1,2$. There is no way to add the two additional generators in a way that increases the rank in these maximal gradings while maintaining the needed spectral sequences and symmetry.

The second possibility is adding a generator in the $(a_1 , a_2)$ grading $(5/2 , 5/2)$ and adding on in the grading $(-5/2 , -5/2)$ to preserve the symmetry. The partial complex with these generators is shown in Figure~\ref{xfivehalf8gens}. At this point there are eight generators in the complex and up to four more that can be placed.

If we add generators with $a_i$ grading larger than $5/2$ then the requirements about even parity in every grading and symmetry mean that we must add exactly two generators in that grading (which is now the top grading) and so one of the components is a braid in the complement of the other. So now we can assume that we only add generators whose $a_i$ gradings are less than or equal to $5/2$ in absolute value.

Now there are at most 4 more generators to add. If the end result will have that neither component is a braid in the complement of the other then two those generators must be added at either $a_1 = 5/2, a_2 = -5/2$ or  $a_1 = 5/2, a_2 = 5/2$ and the other two are then added in the appropriate place for symmetry. There are three possible ways to add the two generators, either both in  $a_1 = 5/2, a_2 = -5/2$, both in $a_1 = 5/2, a_2 = 5/2$, or one in $a_1 = 5/2, a_2 = -5/2$ and the other in $a_1 = 5/2, a_2 = 5/2$. For each of these ways of adding generators, it is impossible to add differentials that give the desired $E_\infty$ pages of both spectral sequences. So then the link $L$ must have that one of its components is a braid in the complement of the other.

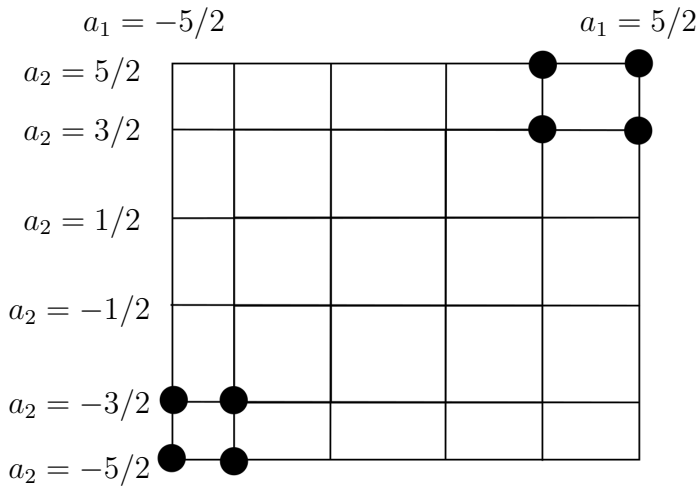
\begin{figure}
\centering{
\begin{overpic}[width= 0.4\textwidth]{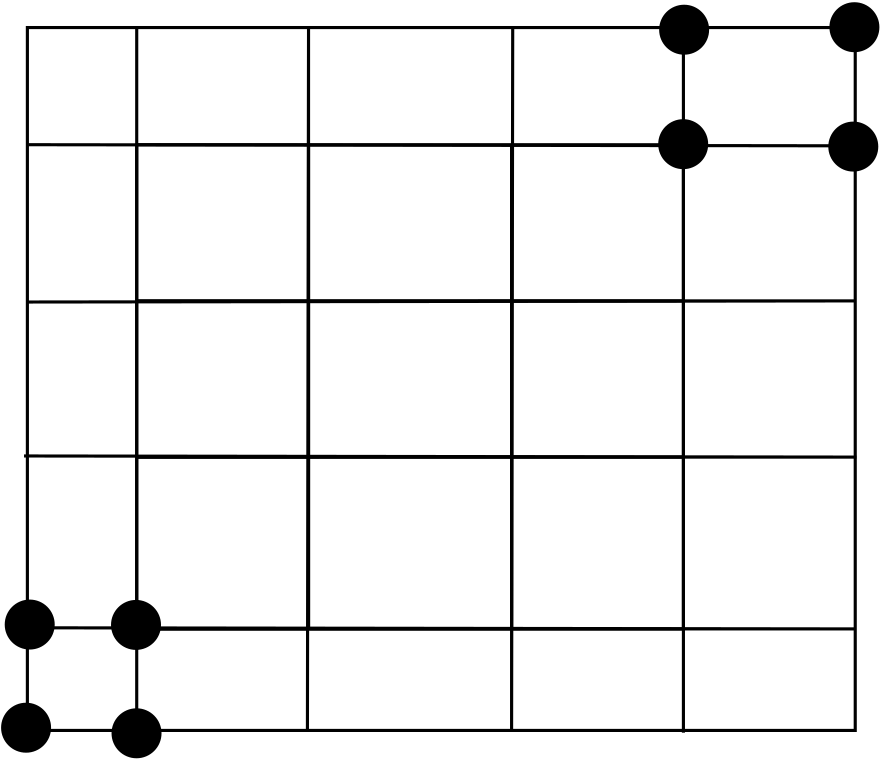}
\put(85,90){$a_1 = 5/2$}
\put(-15, 90){$a_1 = -5/2$}
\put(-27, 80){$a_2 = 5/2$}
\put(-27, 68){$a_2 = 3/2$}
\put(-27, 50){$a_2 = 1/2$}
\put(-30, 32){$a_2 = -1/2$}
\put(-30, 13){$a_2 = -3/2$}
\put(-30, 0){$a_2 = -5/2$}
\end{overpic}
\caption{A partial $\widehat{\HFL}(L)$ complex when $x = 5/2$ with 8 generators. The dots represent bi-degrees where the partial complex has rank one. The vertical columns represent $a_1$ gradings with grading increasing by 1 from $-5/2$ to $5/2$.}\label{xfivehalf8gens}
}
\end{figure}

\emph{The case $ x = 1/2$}

Setting $x = 1/2$ we have that $\widehat{\HFL}(L)$ has rank at least 1 in the tri-gradings $( 0 ,3/2,  1/2)$, $( 0 ,1/2,  3/2)$, $(- 1 ,  -1/2, 3/2)$ and $(- 1 ,  3/2, -1/2)$ and these generators survive in one of the spectral sequences induced by $a_i$. Additionally, $\widehat{\HFL}(L)$ also has rank at least 1 in gradings $( -4 , -3/2,  -1/2)$, $( -4 ,-1/2,  -3/2)$, $(- 3 ,  1/2, -3/2)$ and $(- 3 ,  -3/2, 1/2)$  and these generators do not survive in either spectral sequence. The partial complex with these six generators is shown in Figure~\ref{xonehalf8gens}.

If we add generators with $a_i$ grading larger than $3/2$ then the requirements about even parity in every grading and symmetry mean that we must add exactly two generators in that grading (which is now the top grading) and so one of the components is a braid in the complement of the other. So now we can assume that we only add generators whose $a_i$ gradings are less than or equal to $3/2$ in absolute value. 

Having the appropriate $E_\infty$ pages of the two spectral sequences now means that there must be a generator in each of the following four gradings $(-3, -1/2 , -1/2 )$, $(-1, 1/2 , 1/2 )$, $(-2, 1/2 , -1/2 )$, and $(-2, -1/2 , 1/2 )$. After adding these four generators the rank of the complex is 12 and there are no more generators to add. There are exactly two generators in the top $a_1$ grading and so one of the components is a braid in the complement of the other.

\begin{figure}
\centering{
\begin{overpic}[width= 0.4\textwidth]{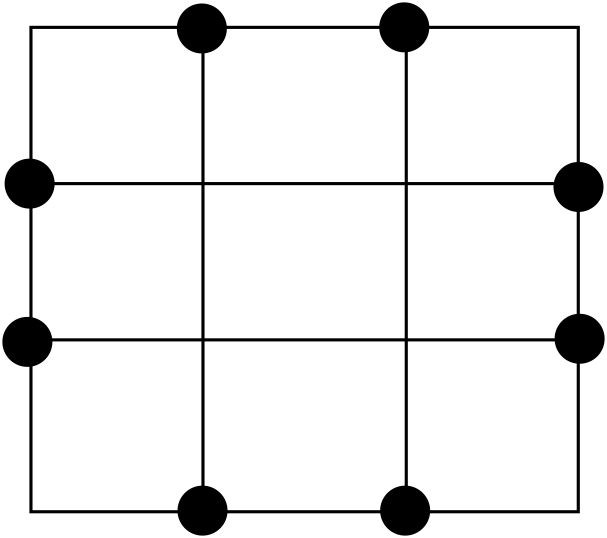}
\put(85,100){$a_1 = 3/2$}
\put(55,100){$a_1 = 1/2$}
\put(20,100){$a_1 = -1/2$}
\put(-15, 100){$a_1 = -3/2$}
\put(-27, 84){$a_2 = 3/2$}
\put(-27, 56){$a_2 = 1/2$}
\put(-30, 28){$a_2 = -1/2$}
\put(-30, 0){$a_2 = -3/2$}
\end{overpic}
\caption{A partial $\widehat{\HFL}(L)$ complex when $x = 1/2$ with 8 generators. The dots represent bi-degrees where the partial complex has rank one.}\label{xonehalf8gens}
}
\end{figure}
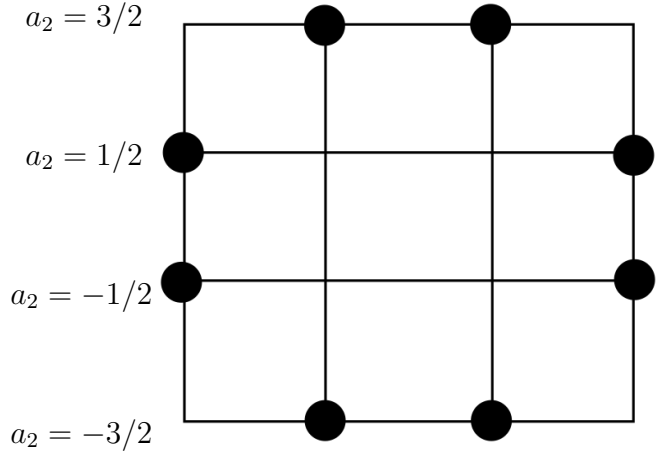

\emph{The case $x > 5/2$}

When $x > 5/2$ the tri-gradings $(0, 3/2 , x)$, $(-1 , 3/2 , x-1)$, $(0, x, 3/2 )$, $(-1  , x-1, 3/2 )$, $(-3 - 2x , -3/2 , -x )$, $(-3 - 2x+1 , -3/2 , -x+1  )$, $(-3 - 2x, -x , -3/2  )$,  and $(-3 - 2x+1 , -x+1 , -3/2  )$ are all distinct and so represent eight different generators of $\widehat{\HFL}(L)$. The only way to add four generators  so that all the $a_i$ gradings have even rank are by adding two at $(a_1 , a_2)$ gradings $(x,x)$ and $(x - 1 , x-1)$ or at $(x , x-1)$ $(x - 1 , x)$. The Maslov gradings are determined by $\widehat{\HFL}(L)$ being supported in a single $\delta$ grading. The other two generators are then added to the appropriate gradings to maintain symmetry. After adding these four generators the rank of the complex is 12 and there are no more generators to add. There are exactly two generators in the top $a_1$ grading and so one of the components is a braid in the complement of the other.

\emph{All other cases}

The arguments to show braidedness in the remaining cases are almost identical to the cases shown and are not repeated. The argument for the case $x<-3/2$ is similar to the case $x> 5/2$. The argument for the case $x= -3/2$ is similar to the case $x=  5/2$. Finally the argument for the case $x=-1/2$ is similar to the case $x= 3/2$. 
\end{proof}

\bibliography{KHDetection}

\end{document}